\documentclass[amscd,amssymb,verbatim]{amsart}
\usepackage{amssymb,amsfonts,amsmath,amscd}
\usepackage[all,2cell]{xy}
\UseAllTwocells

\theoremstyle{plain}
\newtheorem{Thm}{Theorem}[section]
\newtheorem{Lem}[Thm]{Lemma}
\newtheorem{Prop}[Thm]{Proposition}
\newtheorem{Cor}[Thm]{Corollary}

\theoremstyle{definition}
\newtheorem{Def}[Thm]{Definition}
\newtheorem{Ex}[Thm]{Example}

\newtheorem{Numb}[Thm]{}

\theoremstyle{remark}

\newcommand{\C}{\mathbb{C}}

\newcommand{\cG}{{\mathcal G}}
\newcommand{\cH}{{\mathcal H}}

\newcommand{\M}{\mathbb{M}}

\newcommand{\Q}{\mathbb{Q}}

\newcommand{\T}{\mathbb{T}}

\newcommand{\Z}{\mathbb{Z}}

\newcommand{\limind}{\lim_{\rightarrow}}

\newcommand{\Tor}{{\mathrm{Tor}}}

\begin{document}
\title[Lifting Commutation Relations]{Lifting Commutation Relations in Cuntz Algebras}
\author{Bruce Blackadar}
\address{Department of Mathematics/0084 \\ University of Nevada, Reno \\ Reno, NV 89557, USA}
\email{bruceb@unr.edu}
%\thanks{Supported by NSF grant DMS-0070763}

%\keywords{C*-algebra, partial isometry, Cuntz-Krieger algebra}
%\subjclass{Primary: 46L05; Secondary: 19A10}
\date{\today}

\maketitle
\begin{abstract}
We examine splitting of the quotient map from the full free product $A*B$, or the unital free product
$A*_{\C}B$, to the (maximal) tensor product $A\otimes B$, for unital C*-algebras $A$ and $B$.
Such a splitting is very rare, but we show there is one if $A$ and $B$ are both the Cuntz algebra 
$O_2$ or $O_\infty$, and in a few other cases.  The splitting is not explicit (and in principle probably 
cannot be).  We also describe severe $K$-theoretic obstructions to a splitting.
\end{abstract}

\section{Introduction}

Lifting commutation relations from quotients of C*-algebras is a difficult and often unsolvable problem.
We consider the essentially generic case where $A$ and $B$ are C*-algebras (to avoid unnecessary
technicalities, we will only consider the case of unital $A$ and $B$), and we have the quotient map from
the full free product or full unital free product to the tensor product.  

Recall that the full free product $A*B$ is the universal C*-algebra generated by copies of $A$ and $B$ with no
relations (note that this free product is nonunital even if $A$ and $B$ are unital), 
the unital free product $A*_{\C}B$ is the universal C*-algebra generated by copies 
of $A$ and $B$ with
a common unit, and the tensor product $A\otimes B$ is the universal C*-algebra generated by 
commuting copies of $A$ and $B$ with a common unit (all tensor products in this paper are maximal; our examples are nuclear, so this is not much of an issue).  There is a natural quotient map $\pi$ from the
full free product $A*B$, or the unital free product $A*_{\C}B$, to the tensor product $A\otimes B$.

We consider the question of whether there is a splitting (cross section) for this quotient map, i.e.\
a *-homomorphism $\sigma:A\otimes B\to A*B$ (or to $A*_{\C}B$, not necessarily unital) with 
$\pi\circ\sigma$ the identity
on $A\otimes B$.  The short answer is ``rarely''.

Existence of a splitting for the quotient map from $A*B$ to $A\otimes B$ is equivalent to having a
functorial procedure for beginning with two (not necessarily unital) *-homomorphisms 
$\phi:A\to D$ and $\psi:B\to D$ of $A$ and $B$ into a C*-algebra $D$ and manufacturing
new *-homomorphisms $\tilde\phi:A\to D$ and $\tilde\psi:B\to D$ such that $\tilde\phi(A)$ and
$\tilde\phi(B)$ commute and have a common unit.  ``Functorial'' means:
\begin{enumerate}
\item[(i)]  If $f:D\to D'$ is a *-homomorphism, then $\widetilde{f\circ\phi}=f\circ\tilde\phi$ and
$\widetilde{f\circ\psi}=f\circ\tilde\psi$.
\item[(ii)]  If $\phi(A)$ and $\psi(B)$ already commute and have a common unit, then $\tilde\phi=\phi$
and $\tilde\psi=\psi$.
\end{enumerate}
[Such $\phi$ and $\psi$ define a *-homomorphism $\phi*\psi$ from $A*B$ to $D$; let $\tilde\phi$
and $\tilde\psi$ be defined by $(\phi*\psi)\circ\sigma$ for a splitting $\sigma$].  Such a functorial
procedure is automatically point-norm continuous in the sense that if $\phi_n:A\to D$ and
$\psi_n:B\to D$ converge point-norm to $\phi$ and $\psi$ respectively, then $\tilde\phi_n\to\tilde\phi$
and $\tilde\psi_n\to\tilde\psi$ point-norm respectively [it suffices to note that $\phi_n*\psi_n\to
\phi*\psi$ point-norm].

Existence of a splitting is also equivalent to being able to always solve the following lifting problem:
given a C*-algebra $D$, a (closed) ideal $J$ of $D$, and *-homomorphisms $\bar\phi:A\to D/J$
and $\bar\psi:B\to D/J$ such that $\bar\phi(A)$ and $\bar\psi(B)$ commute and have a common unit,
and $\bar\phi$ and $\bar\psi$ separately lift to *-homomorphisms $\phi:A\to D$ and $\psi:B\to D$,
find lifts $\tilde\phi:A\to D$ and $\tilde\psi:B\to D$ such that $\tilde\phi(A)$ and $\tilde\phi(B)$
commute and have a common unit.

\bigskip

We now give examples suggesting that splittings are ``rare''.

\begin{Ex}\label{Ex11}
Consider the simplest nontrivial C*-algebra $\C^2$.  There is an explicit description of 
$\C^2*_{\C}\C^2$, which is the universal unital C*-algebra generated by two projections, as
the continuous functions from $[0,1]$ to $\M_2$ which are diagonal at the endpoints (cf.\ \cite[IV.1.4.2]{BlackadarOperator}).
From this description it is easy to see that there is no splitting for the quotient map from 
$\C^2*_{\C}\C^2$ to $\C^2\otimes\C^2$ (which is just evaluation at the endpoints of $[0,1]$).
There is {\em a fortiori} no splitting for the quotient map from $\C^2*\C^2$ to $\C^2\otimes\C^2$,
since this quotient map factors through $\C^2*_{\C}\C^2$.

There is also no splitting for the quotient map from $\C^2*\C^2$ to $\C^2*_{\C}\C^2$ (\ref{Ex4}).
\end{Ex}

\begin{Ex}\label{Ex21}
Here is an easier example.  Let $n>1$.  Then the quotient map from $\M_n*_{\C}\M_n$ to
$\M_n\otimes\M_n$ cannot split (unitally).  For $\M_n*_{\C}\M_n$ has $\M_n$ as a quotient, but there is
no nonzero homomorphism from $\M_n\otimes\M_n\cong\M_{n^2}$ to $\M_n$.  (Actually it cannot
split nonunitally either; cf.\ Theorem \ref{Thm31}, \S4.)

Similarly, if $m$ and $n$ ($>1$) are not relatively prime, there is no splitting for the quotient map from 
$\M_m*_{\C}\M_n$ to $\M_m\otimes\M_n\cong\M_{mn}$, since $\M_m*_{\C}\M_n$ has $\M_p$ as a quotient,
where $p$ is the least common multiple of $m$ and $n$, and $p<mn$.  If $m$ and $n$ are
relatively prime, the question is more delicate, but it seems unlikely that $\M_m*_{\C}\M_n$
contains a unital copy of $M_{mn}$ (there is no such copy if either $m$ or $n$ is prime
\cite{RordamVFree}).  %In any event there can be no splitting for the quotient
%map from $\M_m*\M_n$ to $\M_m\otimes\M_n$, since $\M_m*\M_n$ has $\M_n$ as a quotient ().

These obstructions to a splitting are $K$-theoretic; see section \ref{Sec3} for a discussion.
\end{Ex}

There is somewhat more hope in general of getting a splitting if the free product is replaced by a
``soft tensor product'' $A\circledast_\epsilon B$ where the copies of $A$ and $B$ are required to approximately commute
(there are various ways to make this precise; cf.\ \ref{Sec21}).  In the case of $\C^2\otimes\C^2$, we want to
restrict to the case where the commutator $pq-qp$ of the two generating projections has small
norm, say $\leq\epsilon$ for a specified $\epsilon>0$.  In fact, in this case, as soon as $\epsilon<\frac{1}{2}$
there is a splitting, as is easily seen.  A similar result holds for $\M_m\circledast_\epsilon\M_n$,
where the matrix units of the two matrix algebras $\epsilon$-commute, and more generally for
$A\circledast_\epsilon B$ for $A$ and $B$ finite-dimensional, for small enough $\epsilon$.

\smallskip

But in only slightly more general settings there is no splitting even in the soft tensor product case:  

\begin{Ex}\label{Ex31}
If we regard $C(\T)\otimes C(\T)$ as the universal C*-algebra generated by two commuting unitaries,
and we let $C(\T)\circledast_\epsilon C(\T)$ be the ``soft torus,'' the universal C*-algebra generated by two
unitaries $u$ and $v$ with $\|uv-vu\|\leq\epsilon$ (\cite{ExelSoft}, \cite{ElliottELSoft}, \cite{EilersEFinite}), then there is no splitting for the natural
quotient map from $C(\T)\circledast_\epsilon C(\T)$ to $C(\T)\otimes C(\T)$ for any $\epsilon>0$
(the Voiculescu matrices (\cite{Voiculescu}, \cite{ExelLAlmost}) can be used to show this is impossible).  Thus {\em a fortiori}
there can be no splitting for the quotient map from $C(\T)*_{\C}C(\T)$ to $C(\T)\otimes C(\T)$,
since this quotient map factors through $C(\T)\circledast_\epsilon C(\T)$.
This example can be essentially summarized by saying that $C(\T^2)\cong C(\T)\otimes C(\T)$
is not semiprojective (\ref{SemProp}).
\end{Ex}

\begin{Ex}\label{Ex41}
Not all obstructions to a splitting are $K$-theoretic.  For example, there is no splitting for the
quotient map from $C([0,1])\circledast_\epsilon C([0,1])$ to $C([0,1])\otimes C([0,1])\cong C([0,1]^2)$
for any $\epsilon>0$ since $C([0,1]^2)$ is not semiprojective (\ref{SemProp}), and hence no splitting for
$C([0,1])*_{\C}C([0,1])\to C([0,1])\otimes C([0,1])$ even though there is a splitting on the $K$-theory
level.  (There is, however, a splitting for the quotient map from $C([0,1])*C([0,1])$ to
$C([0,1])*_{\C} C([0,1])$; in fact, there is a simple explicit cross section for the map from
$A*C([0,1])$ to $A*_{\C}C([0,1])$ for any unital $A$.)
\end{Ex}

One can ask whether the quotient map from $A*_{\C}B$ to $A\otimes B$ {\em ever} splits
(if neither $A$ nor $B$ is $\C$).
Perhaps surprisingly, the answer is yes: it splits if $A$ and $B$ are certain Kirchberg algebras
(but far from all pairs of Kirchberg algebras).  For example, in one of our main results 
we show (Corollary \ref{KirchCor}) it splits if $A=B=O_2$
or if $A$ is any semiprojective Kirchberg algebra and $B=O_\infty$.  In fact, even the quotient
map from $A*B$ to $A\otimes B$ splits in these cases.  It has been known that in each of these
cases, the quotient map from $A\circledast_\epsilon B$ to $A\otimes B$ splits 
if $\epsilon>0$ is sufficiently small, where $A\circledast_\epsilon B$ denotes the universal
C*-algebra in which the standard generators and their adjoints in the two algebras 
$\epsilon$-commute (this is easily proved from the definition of semiprojectivity, cf.\ \ref{SemProp});
but it came as a surprise to the author that there is a splitting even when no commutation condition
(or any other relation) between the algebras is assumed.

See Section \ref{Sec2} for definitions and a brief discussion of semiprojectivity and Kirchberg algebras,
and \cite{BlackadarOperator} for general information about C*-algebras.  \cite{LoringLifting}
contains an extensive discussion of lifting problems and semiprojectivity for C*-algebras.

I am indebted to the referee for pointing out a gap in the proof of Theorem \ref{Thm31}, and for 
comments leading to Theorem \ref{Thm40} and an improvement in Theorem \ref{Thm52}.  I also
thank Mikael R{\o}rdam and Wilhelm Winter for helpful conversations.

\section{Preliminaries}\label{Sec2}

\begin{Numb}\label{Sec21}
If $A$ and $B$ are separable and unital, and $\cG$ and $\cH$ are sets of generators for $A$ and $B$ 
respectively which are finite or sequences converging to 0, and $\epsilon>0$, we define
$A\circledast_\epsilon B$ to be the universal unital C*-algebra generated by unital copies of
$A$ and $B$ such that $\|[a,b]\|\leq\epsilon$ for all $a\in\cG$, $b\in\cH$.  The notation does
not reflect the dependence on $\cG$ and $\cH$, which is not too important qualitatively;
if we instead write $A\circledast_{\epsilon,\cG,\cH}B$, and $\cG'$ and $\cH'$ are different
sequences of generators converging to 0, then for any $\epsilon>0$ there is an $\epsilon'>0$
such that the natural quotient map from $A\circledast_{\epsilon,\cG,\cH}B$ to $A\otimes B$
factors through $A\circledast_{\epsilon',\cG',\cH'}B$.  We will thus fix $\cG$ and $\cH$ and 
just write $A\circledast_\epsilon B$.

The natural quotient map from $A\circledast_\epsilon B$ to $A\otimes B$ factors through
$A\circledast_{\epsilon'}B$ for any $\epsilon'<\epsilon$.
\end{Numb}

\begin{Numb}
Recall the definition of semiprojectivity (\cite{BlackadarShape}, \cite[II.8.3.7]{BlackadarOperator}):
A separable C*-algebra $A$ is {\em semiprojective} if, whenever $D$ is a C*-algebra, $(J_n)$ an increasing sequence of closed (two-sided) ideals of $D$, and $J=[\cup J_n]^-$, then any homomorphism $\phi:A\to D/J$ can be partially lifted to a homomorphism
$\psi:A\to D/J_n$ for some sufficiently large $n$.

If $A$ is unital, then in checking semiprojectivity from the definition $D$ and $\phi$ may be 
chosen unital, and then the partial lift $\psi$ can be required to be unital (in fact, $\psi$ will
automatically become unital if $n$ is sufficiently increased).

There are many known examples of semiprojective C*-algebras, but semiprojectivity is quite
restrictive.  For example, if $X$ is a compact metrizable space, $C(X)$ is semiprojective
if and only if $X$ is an ANR of dimension $\leq1$ \cite{SorensenTCharacterization}; the
dimension restriction is essentially exactly the fact that commutation relations cannot be
partially lifted in general.  See also \ref{Sec25}.
\end{Numb}

\begin{Prop}\label{SemProp}
Let $A$ and $B$ be unital semiprojective C*-algebras.  Then the quotient map from 
$A\circledast_\epsilon B$ to $A\otimes B$ splits (unitally) for some $\epsilon>0$ (hence for all
sufficiently small $\epsilon>0$) if and only if $A\otimes B$ is semiprojective.
\end{Prop}

\begin{proof}
Suppose $A\otimes B$ is semiprojective.  Set $D=A*_{\C}B$ and $J_n$ the kernel of the
natural quotient map from $D$ onto $A\circledast_{1/n}B$.  Then $(J_n)$ is an increasing
sequence of closed ideals of $D$, and if $J=[\cup_n J_n]^-$, then $D/J\cong A\otimes B$.
By semiprojectivity there is a (unital) partial lift of the identity map on $A\otimes B$ for some $n$.

Conversely, suppose there is a splitting $\sigma$ for the quotient map from $A\circledast_\epsilon B$ 
to $A\otimes B$ for some
$\epsilon>0$.  Let $D$, $(J_n)$, $\phi$ be as in the definition of semiprojectivity, with $D$
and $\phi$ unital.  For some sufficiently large $n$ there are unital partial lifts $\psi_A$ of
$\phi|_{A\otimes 1}$ and $\psi_B$ of $\phi|_{1\otimes B}$, which give a unital homomorphism
$\psi$ from $A*_{\C}B$ to $D/J_n$ which is a lift of $\phi\circ\pi$, where $\pi$ is the quotient map
from $A*_{\C}B$ to $A\otimes B$ (this is essentially the argument that shows that $A*_{\C}B$
is semiprojective).  Since the partial lifts of the generators of $A$ and $B$ asymptotically
commute as $n\to\infty$, if $n$ is sufficiently increased, the map $\psi$ factors through 
$A\circledast_\epsilon B$.  Then $\psi\circ\sigma$ is a partial lift of $\phi$, so $A\otimes B$
is semiprojective.
\end{proof}

\begin{Numb}\label{Kirch}
A {\em Kirchberg algebra} is a separable nuclear purely infinite (simple unital) C*-algebra.
It is not known whether such a C*-algebra is automatically in the bootstrap class for the Universal Coefficient
Theorem \cite[22.3.4]{BlackadarKTheory}; a Kirchberg algebra in this bootstrap class is called a {\em UCT Kirchberg algebra}.

The first Kirchberg algebras to be studied were the {\em Cuntz algebras} $O_n$, $2\leq n\leq\infty$.
If $2\leq n<\infty$, $O_n$ is the universal (unital) C*-algebra generated by $n$ isometries with
mutually orthogonal range projections adding up to the identity.  $O_\infty$ is the universal (unital)
C*-algebra generated by a sequence of isometries with mutually orthogonal range projections.
The next class was the (simple) Cuntz-Krieger algebras $O_A$.  These are all UCT Kirchberg algebras.

Kirchberg (in part done also independently by Phillips) showed that the UCT Kirchberg algebras
are classified by their $K$-theory.  If $A$ is a unital C*-algebra, write $L(A)$ for the triple
$(K_0(A),K_1(A),[1_A])$; $L(A)$ is part of the Elliott invariant of $A$, and is the entire Elliott
invariant of $A$ if $A$ is a Kirchberg algebra.  A unital *-homomorphism $\phi$ from $A$ to $B$ induces
a morphism from $L(A)$ to $L(B)$, i.e.\ a homomorphism $\phi_{\ast0}:K_0(A)\to K_0(B)$ with
$\phi_{\ast0}([1_A])=[1_B]$ and a homomorphism $\phi_{\ast1}:K_1(A)\to K_1(B)$. The following
facts are known (see \cite{Rordam} for a good exposition):
\begin{enumerate}
\item[(i)]  If $A$ and $B$ are UCT Kirchberg algebras and $\alpha:L(A)\to L(B)$ is a morphism,
then there is a unital *-homomorphism $\phi:A\to B$ with $\phi_\ast=\alpha$.  If $\alpha$ is an
isomorphism, $\phi$ can be chosen to be an isomorphism.  In particular, $L(A)$ is a complete
isomorphism invariant of $A$ among UCT Kirchberg algebras.
\item[(ii)]  If $G_0$ and $G_1$ are any countable abelian groups (recall that the $K$-groups of any
separable C*-algebra are countable), and $u$ is any element of $G_0$, then there is a
UCT Kirchberg algebra $A$ (unique up to isomorphism) with $L(A)\cong(G_0,G_1,u)$.
\end{enumerate}

Cuntz showed that $L(O_n)=(\Z_{n-1},0,1)$ if $n<\infty$ and $L(O_\infty)=(\Z,0,1)$.  As technical
results toward the above classification, Kirchberg (\cite{KirchbergPEmbedding}; cf.\ \cite{Rordam}) showed the following facts about tensor products:
\begin{enumerate}
\item[(iii)]  For any Kirchberg algebra $A$ (UCT or not), $O_2\otimes A\cong O_2$.  In particular, 
$O_2\otimes O_2\cong O_2$.
\item[(iv)]  For any Kirchberg algebra $A$ (UCT or not), $O_\infty\otimes A\cong A$.  In particular, 
$O_\infty\otimes O_\infty\cong O_\infty$.
\item[(v)]  The infinite tensor products $O_2^\infty$ and $O_\infty^\infty$ are isomorphic to $O_2$
and $O_\infty$ respectively.
\end{enumerate}

The isomorphism $O_2\otimes O_2\cong O_2$ (which was first shown by Elliott; cf.\ \cite{RordamShort}) and the others 
are quite deep results.  They are highly
nonconstructive; in fact, it is in principle essentially impossible to give an explicit isomorphism
of $O_2\otimes O_2$ and $O_2$ by the results of \cite{AraC}.
\end{Numb}

\begin{Numb}\label{Sec25}
It has recently been shown \cite{Enders} that if $A$ is a UCT Kirchberg algebra, then $A$ is
semiprojective if and only if $K_\ast(A)$ is finitely generated (in fact, this problem was the original
motivation for the work of the present paper).  Special cases were previously known: it is easy
to show that $O_n$ ($n<\infty$) and, more generally, $O_A$ for any $A$ is semiprojective
\cite{BlackadarShape}.  $O_\infty$ was shown to be semiprojective in \cite{BlackadarSemiprojectivity},
and more general results were obtained in \cite{Szymanski} and \cite{Spielberg}.
\end{Numb}

\section{$K$-Theoretic Obstructions, Unital Free Product Case}\label{Sec3}

It turns out that there are rather severe $K$-theoretic obstructions to a splitting
for the quotient map from $A*_{\C}B$ to $A\otimes B$.  In this section,
we will restrict attention to separable nuclear C*-algebras in the UCT class \cite[22.3.4]{BlackadarKTheory},
so that the K\"{u}nneth Theorem for Tensor Products \cite[23.1.3]{BlackadarKTheory} and the results of
\cite{Germain} hold, although some 
of what we do here works in greater generality.  The analysis is largely an elementary (but moderately
complicated) exercise in group theory, so some details are omitted.

Let $A$ and $B$ be unital C*-algebras of the above form.  We recall the results of the K\"{u}nneth
theorem and of \cite{Germain}:

\smallskip

%\begin{enumerate}
%\item[] 
$$K_0(A\otimes B)\cong [K_0(A)\otimes_{\Z}K_0(B)]\oplus[K_1(A)\otimes_{\Z}K_1(B)]
\oplus\Tor_1^{\Z}(K_0(A),K_1(B))\oplus\Tor_1^{\Z}(K_1(A),K_0(B))$$
%\item[] 
$$K_1(A\otimes B)\cong [K_0(A)\otimes_{\Z}K_1(B)]\oplus[K_1(A)\otimes_{\Z}K_0(B)]
\oplus\Tor_1^{\Z}(K_0(A),K_0(B))\oplus\Tor_1^{\Z}(K_1(A),K_1(B))$$
%\end{enumerate}

\smallskip

\noindent
where $\Tor_1^{\Z}$ denotes the Tor-functor of homological algebra.

\smallskip

From the exact sequence of \cite{Germain}, we obtain:

\smallskip

%\begin{enumerate}
%\item[] 
$$K_0(A*B)\cong K_0(A)\oplus K_0(B)$$
%\item[] 
$$K_1(A*B)\cong K_1(A)\oplus K_1(B)$$
%\end{enumerate}

\smallskip

%\begin{enumerate}
%\item[] 
$$K_0(A*_{\C}B)\cong[K_0(A)\oplus K_0(B)]/\langle([1_A],-[1_B])\rangle$$

\smallskip

%\item[] 
$$K_1(A*_{\C}B)\cong K_1(A)\oplus K_1(B)\oplus\Z$$
if $[1_A]$ and $[1_B]$ are torsion elements of $K_0(A)$ and $K_0(B)$ respectively, and
$$K_1(A*_{\C}B)\cong K_1(A)\oplus K_1(B)$$
otherwise.
%\end{enumerate}

%\smallskip

In $K_1(A*_{\C}B)$, the first two summands are well defined as subsets of the group, but the
extra $\Z$ summand is not.

\smallskip

The quotient map from $A*B$ to $A*_{\C}B$ induces the obvious maps on the $K$-groups.

The quotient map $\pi$ from $A*_{\C}B$ to $A\otimes B$ induces the following maps:
on $K_0$, $\pi_{\ast0}(x,y)=x\otimes[1_B]+[1_A]\otimes y$ in the $K_0(A)\otimes_{\Z}K_0(B)$ summand; on the first two summands of $K_1$,
$\pi_{\ast1}(x,y)=([1_A]\otimes y)\oplus(x\otimes[1_B])$ in the $[K_0(A)\otimes_{\Z}K_1(B)]\oplus[K_1(A)\otimes_{\Z}K_0(B)]$ summands.  

If there is a splitting, there must be cross sections for these maps.  There will be a $K$-theoretic
obstruction almost any time both $A$ and $B$ have nontrivial $K_1$, or if the rank of $K_0$
of both $A$ and $B$ is at least 2. 

We first dispose of the extra summand of $\Z$ in $K_1$ when it occurs.
One can give a more precise and detailed analysis of this summand, but the following argument
suffices for the purposes of this paper.

\begin{Lem}
If $[1_A]$ and $[1_B]$ are torsion, the extra summand $\Z$ in $K_1(A*_{\C}B)$ can be chosen
so that the restriction $\phi$ of the quotient map to this summand maps into 
$\Tor_1^{\Z}(K_0(A),K_0(B))$.
\end{Lem}

\begin{proof}
Suppose $m[1_A]=0$ in $K_0(A)$ and $n[1_B]=0$ in $K_0(B)$ for natural numbers $m,n$.  
Since the exact sequence of
\cite{Germain} is natural (functorial), it does not change if we replace $A$ and $B$ by 
$A\otimes O_\infty$ and $B\otimes O_\infty$ respectively, i.e.\ we may assume $A$ and $B$ are
properly infinite.  Thus $A$ and $B$ contain unital copies of $O_{m+1}$ and $O_{n+1}$ respectively.
There is a corresponding summand of $\Z$ in $K_1(O_{m+1}*_{\C}O_{n+1})$, which is well defined
as the entire group $K_1(O_{m+1}*_{\C}O_{n+1})$, and the $Z$ summand in $K_1(A*_{\C}B)$
can be taken to be the image of this summand, which maps into 
$\Tor_1^{\Z}(K_0(O_{m+1}),K_0(O_{n+1}))$ since this is all of $K_1(O_{m+1}\otimes O_{n+1})$.
\end{proof}

\begin{Cor}
There can be a splitting at the $K$-theory level only if $\phi$ is zero, i.e.\ the image of $K_1(A*_{\C}B)$
in $K_1(A\otimes B)$ is the same as the image of the subgroup $K_1(A)\oplus K_1(B)$.
\end{Cor}

\begin{proof}
$\Tor_1^{\Z}(K_0(A),K_0(B))$ is a torsion group, so $\phi$ cannot be injective.  A quotient map
of $\Z$ can only split if it is zero or injective.
\end{proof}

Thus we are reduced to considering only the summands $K_1(A)\oplus K_1(B)$ of $K_1(A*_{\C}B)$.

\begin{Prop}\label{KthObsProp}
A splitting for the quotient map from $A*_{\C}B$ to $A\otimes B$ is impossible under any of the 
following conditions:
\begin{enumerate}
\item[(i)]  $K_1(A)\otimes_{\Z}K_1(B)\neq0$.
\item[(ii)]  $(rank(K_0(A)))(rank(K_0(B)))>rank(K_0(A))+rank(K_0(B))$, or 
\linebreak
$rank(K_0(A))+rank(K_0(B))-1$ if $[1_A]$ or $[1_B]$ has infinite order.
\item[(iii)]  $\Tor_1^{\Z}(K_\ast(A),K_\ast(B))\neq0$.
\end{enumerate}
\end{Prop} 

This is only the beginning of the $K$-theoretic obstructions.  We will analyze the case where
$K_\ast(A)$ and $K_\ast(B)$ are finitely generated.  Then we have
$$K_0(A)=\Z^a\oplus G_0,\ \ \ K_1(A)=\Z^n\oplus G_1$$
$$K_0(B)=\Z^b\oplus H_0,\ \ \ K_1(B)=\Z^m\oplus H_1$$
for nonnegative integers $a,b,n,m$ and finite abelian groups $G_0,G_1,H_0,H_1$.

In order to have a splitting, from \ref{KthObsProp}(iii) we need the orders of $G_0$ and $G_1$ to be
relatively prime to the orders of $H_0$ and $H_1$, so we will assume this from now on.  We also have
that $K_1(A)\otimes_{\Z}K_1(B)=0$ is necessary for a splitting by \ref{KthObsProp}(i), which implies
that $n$ and $m$ cannot both be nonzero; without loss of generality we assume $m=0$.  We must
have $\Z^n\otimes_{\Z} H_1\cong H_1^n=0$, so either $n=0$ or $H_1=0$.  %Also,
%we must have $\Z^n\otimes_{\Z}H_0\cong H_0^n=0$, so either $n=0$ or $H_0=0$.  
We then obtain:

\smallskip

\begin{equation}
K_0(A*_{\C}B)=(\Z^{a+b}\oplus G_0\oplus H_0)/\langle[1_A],-[1_B]\rangle
\end{equation}
\begin{equation}
K_0(A\otimes B)=\Z^{ab}\oplus G_0^b\oplus H_0^a
\end{equation}

\smallskip

\begin{equation}
K_1(A*_{\C}B)=\Z^n\oplus G_1\oplus H_1
\end{equation}
\begin{equation}
K_1(A\otimes B)=\Z^{nb}\oplus G_1^b\oplus H_1^a\oplus H_0^n
\end{equation}

 \smallskip
 
\noindent
since $G_i\otimes_{\Z}H_j=0$ for all $i,j$. 

We need for the induced maps from $K_i(A*_{\C}B)$ to $K_i(A\otimes B)$ to be surjective
and have a cross section.  For this, we need $H_0^n=0$, so $n=0$ or $H_0=0$.

\begin{Numb}
There are two trivial cases: $K_\ast(A)$ is either $(0,0)$, or $(\Z,0)$ and $[1_A]=1$, in which case there is no
restriction on $K_\ast(B)$.  (The situation is, of course, symmetric in $A$ and $B$.)
\end{Numb} 

\begin{Numb}
Next we can easily dispose of the case where $n>0$.  In this case we observed earlier that $H_0=H_1=0$,
and from (4) we see that $b\leq1$.  Thus $K_\ast(B)$ is either $(0,0)$ or $(\Z,0)$.
Suppose $K_0(B)=\Z$.  The map $\pi_{\ast1}$ from $K_1(A*_{\C}B)\cong\Z^n\oplus G_1$ to 
$K_1(A\otimes B)\cong\Z\oplus G_1$
is multiplication by $[1_B]$; for this to be surjective, we need for $[1_B]$ to be a generator of $K_0(B)$.
Thus we are in one of the trivial cases.
\end{Numb}

\begin{Numb}
From now on we assume $n=0$, i.e.\ $K_1(A)$ and $K_1(B)$ are finite groups.  The situation is
now symmetric in $A$ and $B$.  The map from $K_0(A*_{\C}B)$ to $K_0(A\otimes B)$
must send $\Z^{a+b}/\langle[1_A],-[1_B]\rangle$ onto $\Z^{ab}$, and thus we must have
$ab\leq a+b$, and if either the torsion-free part of $[1_A]$ or $[1_B]$ is nonzero, we must have
$ab\leq a+b-1$.  Thus we have that either $a$ or $b$ is $\leq1$, or $a=b=2$ and the torsion-free
part of both $[1_A]$ and $[1_B]$ is 0.  But in the case $a=b=2$, the maps $\pi_{\ast0}$ 
on the torsion-free part of $K_0$ are zero and cannot be surjective.  Thus there is no
splitting in this case.
\end{Numb} 

\begin{Numb}
Now assume without loss of generality that $a=0$ or 1.  
We first dispose of the cases where $[1_A]$ or $[1_B]$ has finite order (which includes the
case $a=0$).  Suppose
$$(K_0(A),K_1(A),[1_A])=(\Z\oplus G_0,G_1,(u,r))$$
$$(K_0(B),K_1(B),[1_B])=(\Z^b\oplus H_0,H_1,(v,s))$$
Then
$$K_\ast(A*_{\C}B)=((\Z\oplus G_0\oplus\Z^b\oplus H_0)/\langle(u,r,v,s)\rangle,G_1\oplus H_1)$$
$$K_\ast(A\otimes B)=(\Z^b\oplus G_0^b\oplus H_0,G_1^b\oplus H_1)$$
If $u=0$, then $\pi_{\ast0}$ maps $\Z^b\oplus H_0$ to 0, so for the map to be surjective we
must have $b\leq1$ and $H_0=0$; if $b=1$, for $\pi_{\ast0}$ to map $\Z\oplus G_0$
onto $K_0(A\otimes B)$ we must have $v=1$.  Similarly, $\pi_{\ast1}$ maps the $H_1$ to 0, so $H_1=0$.  Thus we are in a trivial case (for $B$).

\noindent
If $v=0$, we argue similarly that $G_0=G_1=0$ and $u=1$, so we are again in a trivial case
(for $A$).

Now consider the case
$$(K_0(A),K_1(A),[1_A])=(G_0,G_1,r)$$
$$(K_0(B),K_1(B),[1_B])=(\Z^b\oplus H_0,H_1,(v,s))$$
where $b\geq1$ and $v=0$.  Then as before $G_0=G_1=0$ and we are in a trivial case (for $A$).

The last case is where both $K_\ast(A)$ and $K_\ast(B)$ are torsion (finite) groups.  Then
$$K_\ast(A\otimes B)=(0,0)$$
and there is no $K$-theoretic restriction.
\end{Numb}

\begin{Numb}
The remaining case is where $a=1$, $b\geq1$, and $[1_A]$ and $[1_B]$ have infinite order.
This is the most delicate case.
Write $u$ and $v$ for the components of $[1_A]$ and $[1_B]$ in the torsion-free parts of
$K_0(A)$ and $K_0(B)$ respectively; we may assume without loss of generality that $u>0$
and that $v=(w,0,\dots,0)$ with $w>0$.

First consider the case $b=1$.  We have
$$K_0(A*_{\C}B)\cong(\Z\oplus G_0\oplus\Z\oplus H_0)/\langle(u,r,-w,-s)\rangle\cong
\Z^2/\langle (u,-w)\rangle\oplus G_0\oplus H_0$$
and the map $\pi_{\ast0}$ to $K_0(A\otimes B)=\Z\oplus G_0^b\oplus H_0$ must send
$\Z^2/\langle(u,-w)\rangle$ onto the first coordinate.  This map is multiplication by $w$ on the first
coordinate plus multiplication by $u$ on the second; thus it can only be surjective if $u$ and $w$
are relatively prime.  In this case, it is both injective and surjective on this piece.

The map $\pi_{\ast0}$ maps the $G_0$ to the $G_0$ and is multiplication by $w$.  For this to
be surjective (hence bijective), we must have $wG_0=G_0$, i.e.\ $w$ must be relatively prime to
the order of $G_0$.  Similarly, $u$ must be relatively prime to the order of $H_0$.  The same
argument for $\pi_{\ast1}$ shows that $u$ and $w$ are relatively prime to the orders of
$H_1$ and $G_1$ respectively.

We cannot rule out the case where $u$ and $v$ are greater than 1 and relatively prime.  For
example, consider the case $u=2$, $v=3$.  For simplicity suppose the $K$-groups are torsion-free,
i.e.\ $K_\ast(A)=(\Z,0,2)$ and $K_\ast(B)=(\Z,0,3)$.  Then $K_0(A*_{\C}B)=\Z^2/\langle(2,-3)\rangle$
with order unit $[(2,0)]=[(0,3)]$.  $K_0(A\otimes B)=\Z$ with order unit 6,
and the map $\pi_{\ast0}$ sends $[(x,y)]$ to $3x+2y$, and is an isomorphism;
the inverse map sends $n$ to $[(n,-n)]$.  It is possible that this inverse map is induced
by a homomorphism on the algebra level in some cases (e.g.\ possibly in the case $A=M_2(O_\infty)$
and $B=M_3(O_\infty)$) (\ref{ExMxO}), although it is not in the case $A=\M_2$, $B=\M_3$ (\ref{Ex33}).

Thus the possibilities are
$$L(A)=(\Z\oplus G_0,G_1,(u,r)),\ \ \ L(B)=(\Z\oplus H_0,H_1,(w,s))$$
where $u$ is relatively prime to $w$ and to the orders of $H_0$ and $H_1$, and $w$ is relatively
prime to the orders of $G_0$ and $G_1$.
\end{Numb}

\begin{Numb}
Now consider the case $b>1$.  We have
$$K_0(A*_{\C}B)\cong(\Z\oplus G_0\oplus\Z^b\oplus H_0)/\langle(u,r,-w,0,\dots,0,-s)\rangle$$
$$\cong\Z^2/\langle (u,-w)\rangle\oplus\Z^{b-1}\oplus G_0\oplus H_0$$
and the map $\pi_{\ast0}$ to $K_0(A\otimes B)=\Z^b\oplus G_0^b\oplus H_0$ must send
$\Z^2/\langle(u,-w)\rangle$ onto the first coordinate.  This map is multiplication by $w$ on the first
coordinate plus multiplication by $u$ on the second; thus it can only be surjective if $u$ and $w$
are relatively prime.  In this case, it is both injective and surjective on this piece.

The map $\pi_{\ast0}$ sends $G_0$ into $G_0^b$ and sends $x$ to $(wx,0,\dots,0)$.
This can only be surjective if $G_0=0$.  Similarly, $\pi_{\ast1}$ sends $G_1$ into $G_1^b$
by the same formula, so $G_1=0$.  Thus $K_\ast(A)=(\Z,0)$.

However, we do not need to have $u=1$.  But there is a restriction: the map $\pi_{\ast0}$ sends
$H_0$ into $H_0$  by multiplication by $u$; this needs to be surjective, i.e.\ $u$ is relatively prime
to the order of $H_0$.  Similarly, $uH_1=H_1$, so $u$ is relatively prime to the order of $H_1$.
Thus, for any $b>1$, we have the possibilities
$$L(A)=(\Z,0,u),\ \ \ L(B)=(\Z^b\oplus H_0,H_1,(w,0,\dots,0,s))$$
where $u$ is relatively prime to $w$ and to the orders of $H_0$ and $H_1$. 
\end{Numb} 

\bigskip

We summarize:

\begin{Thm}\label{Thm31}
Let $A$ and $B$ be separable nuclear unital C*-algebras in the UCT class with finitely generated
$K$-theory.  Then there can be a splitting for the quotient map from $A*_{\C}B$ to $A\otimes B$
only in the following situations:
\begin{enumerate}
\item[(i)] $L(A)=(0,0,0)$ or $(\Z,0,1)$, $L(B)$ arbitrary (or vice versa).  In the first case
$L(A\otimes B)=(0,0,0)$, and in the second $L(A\otimes B)\cong L(B)$.
\item[(ii)]  $L(A)=(G_0,G_1,r)$, $L(B)=(H_0,H_1,s)$.  In this case $L(A\otimes B)=(0,0,0)$.
\item[(iii)]  $L(A)=(\Z\oplus G_0,G_1,(u,r))$, $L(B)=(\Z\oplus H_0,H_1,(w,s))$,
where $u$ is relatively prime to $w$ and to the orders of $H_0$ and $H_1$, and $w$ is relatively
prime to the orders of $G_0$ and $G_1$.
\item[(iv)] $L(A)=(\Z,0,u)$, $L(B)=(\Z^b\oplus H_0,H_1,(w,0,\dots,0,s))$,
where $b>1$, and $u$ is relatively prime to $w$ and to the orders of $H_0$ and $H_1$. 
\end{enumerate}
In (ii)--(iv), $G_0,G_1,H_0,H_1$ are any finite abelian groups with the orders of $G_0$ and $G_1$
relatively prime to the orders of $H_0$ and $H_1$.
\end{Thm}

In all cases except (ii), $K_\ast(A*_{\C}B)$ is the same as $K_\ast(A\otimes B)$ and the maps $\pi_{\ast0}$
and $\pi_{\ast1}$ are isomorphisms.  In case (ii), we have 
$$L(A*_{\C}B)=((G_0/rG_0)\oplus(H_0/sH_0),G_1\oplus H_1,0)$$
i.e.\ the identity always has class 0.

Theorem \ref{Thm31} does not guarantee a splitting in these cases; there often is not.
For one thing, there may be more obstructions coming from ordered $K$-theory:

\begin{Ex}\label{Ex33}
(cf.\ Example \ref{Ex21})  Suppose $K_\ast(A)=K_\ast(B)=(\Z,0)$ but $[1_A]=m$, $[1_B]=n$ with $m,n>1$.  Then $K_\ast(A\otimes B)\cong(\Z,0)$ and 
$K_\ast(A*_{\C}B)\cong(\Z^2/\langle(m,-n)\rangle,0)$, so there is potentially a splitting on the group level.
But depending on the actual map $\pi_\ast$ there may not be one at the scaled ordered group level,
or even at the group level.

Suppose $n=m$.  Then $K_0(A*_{\C}B)\cong\Z\oplus\Z_n$, and the order unit in 
$K_0(A*_{\C}B)$ is $(n,0)$.  The order unit in 
$K_0(A\otimes B)$ is $n^2$.  The map $\pi_\ast:K_0(A*_{\C}B)\to K_0(A\otimes B)$ is
multiplication by $n$, which is not surjective, so there can be no cross section.  More generally,
in $n$ and $m$ are not relatively prime, and $d$ is their greatest common divisor and
$p$ is their least common multiple, then $K_0(A*_{\C}B)\cong\Z\oplus\Z_d$ and the order
unit in $K_0(A*_{\C}B)$ is $(p,0)$.  The order unit in 
$K_0(A\otimes B)$ is $mn$; $\pi_\ast$ is multiplication by $\frac{mn}{p}=d$ and is not surjective.

In these cases there is not a splitting at the group level.  But suppose $m$ and $n$ are relatively
prime.  Then there is a splitting at the group level ($\pi_\ast$ is an isomorphism).  
But the positive cone in $K_0(A\otimes B)$ can be larger than the positive cone in the unital free product
$K_0(A*_{\C}B)$, for example if $A=\M_m$, $B=\M_n$; in this case the positive cone in
$K_0(A\otimes B)\cong\Z$ is the usual one, but it is shown in \cite{RordamVFree} that (at least if $m$
or $n$ is prime) the positive cone in $K_0(A*_{\C}B)$ is the subsemigroup of $\Z$ generated by $m$ and $n$.
\end{Ex}

\begin{Ex}\label{ExMxO}
The situation can be nicer with Kirchberg algebras: the unital free product
$M_2(O_\infty)*_{\C}M_3(O_\infty)$ contains a unital copy of $\M_6$.  Let $\{e_{ij}:1\leq i,j\leq3\}$
be the standard matrix units in the $M_3(O_\infty)$.  There are mutually orthogonal subprojections 
$f_1,f_2,f_3,g_1,g_2$ of $e_{11}$ in $M_3(O_\infty)$ adding up to $e_{11}$ such that $f_1,f_2,f_3$ are equivalent to
$e_{11}$ and $g_1$ and $g_2$ are equivalent (take $[f_k]=[e_{11}]$ and $[g_k]=-[e_{11}]$
for all $k$ in $K_0(O_\infty)\cong\Z$).  Then 
$$f_1+f_2+f_3\sim 1_{M_3(O_\infty)}\ .$$
We have that
$1_{M_3(O_\infty)}=1_{M_2(O_\infty)}$ since the free product is unital, and $1_{M_2(O_\infty)}$
is the sum of two equivalent projections in the $M_2(O_\infty)$, so $f_1+f_2+f_3$ is the sum of
two equivalent projections $r_1,r_2$.  Set $p_{11}=r_1+g_1$ and $p_{12}=r_2+g_2$; then
$e_{11}=p_{11}+p_{12}$, and $p_{11}\sim p_{12}$.  For $j=2,3$ set $p_{j1}=e_{j1}p_{11}e_{1j}$
and $p_{j2}=e_{j1}p_{12}e_{1j}$; then 
$$\{p_{jk}:1\leq j\leq3,1\leq k\leq2\}$$
are six mutually
orthogonal equivalent projections in $M_2(O_\infty)*_{\C}M_3(O_\infty)$ adding up to the identity.

This argument generalizes to show that if $m$ and $n$ are relatively prime, then 
$M_m(O_\infty)*_{\C}M_n(O_\infty)$ contains a unital copy of $\M_{mn}$.  Thus there is no
$K$-theoretic obstruction to a splitting in this case.
\end{Ex}

\begin{Ex}
If $A$ or $B$ does not have finitely generated $K$-theory, the situation can be much more
complicated.  For example, if $A$ is the CAR algebra, there does not appear to be any $K$-theoretic
obstruction to a cross section for the quotient map from $A*_{\C}A$ to $A\otimes A$, 
but it is questionable that there is a splitting in this case.
\end{Ex}

One unresolved question is: if there is a splitting for the quotient map from $A*_{\C}B$ to
$A\otimes B$, is there necessarily a unital splitting (and is a splitting even necessarily unital)?
It follows from Theorem \ref{Thm31} and the comment afterward that in the UCT case with finitely
generated $K$-theory, any splitting must necessarily at least send the identity of $A\otimes B$
to a projection in $A*_{\C}B$ whose $K_0$-class is the same as the $K_0$-class of the identity.

\section{$K$-Theoretic Obstructions, General Free Product Case}\label{Sec4}

We can do an essentially identical analysis of obstructions to splitting of the quotient map
from $A*B$ to $A\otimes B$.  But this case can also be handled more easily: if there is a
splitting for this map, there is also a splitting for the quotient map from $A*_{\C}B$ to $A\otimes B$,
so the restrictions in this case are more severe.  However, it is easily seen that in all the
cases in Theorem \ref{Thm31} where a splitting is possible, there is actually a splitting (at
the $K$-theory level) of the quotient map from $A*B$ to $A*_{\C}B$.  So the $K$-theoretic
restrictions are identical in the two cases.  (Restrictions from {\em ordered} $K$-theory
may be more severe in the full free product case; cf.\ Example \ref{Ex21}.)

The quotient map from $A*B$ to $A*_{\C}B$ does not split in general.  There are $K$-theoretic
obstructions in many cases; for example, a quotient map between finite abelian groups
generally does not split.

Ordered $K$-theory also implies nonsplitting of the quotient map from $A*B$ to $A*_{\C}B$
in some cases even when there is a splitting at the $K$-theory level:

\begin{Ex}\label{Ex4}
The positive cone of $K_0(\C^2*\C^2)\cong\Z^2\oplus\Z^2=\Z^4$ is the set of 4-tuples with nonnegative 
entries, and the scale consists of points with each entry 0 or 1 (not all such points; in fact, it can be shown that the scale only consists of $(1,1,0,0)$, $(0,0,1,1)$ and the 4-tuples with at most one 1).  
The positive cone and scale of 
$K_0(\C^2*_{\C}\C^2)\cong\Z^4/\langle(1,1,-1,-1)\rangle$ are the images of the ones in $\Z^4$.
If there is a cross section $\sigma:\C^2*_{\C}\C^2\to\C^2*\C^2$ for the quotient map, $\sigma_\ast$
must send $[(1,0,0,0)]$ to an element of the scale in $\Z^4$ which is in the equivalence class;
the only such element is $(1,0,0,0)$ (in fact, this is the only element of the equivalence class
in the positive cone).  Similarly, we must have $\sigma_\ast([(0,1,0,0)])=(0,1,0,0)$,
etc.  But there is no such homomorphism since then $0=\sigma_\ast(0)=\sigma_\ast([(1,1,-1,-1)])
=(1,1,-1,-1)\neq0$.  Thus there can be no cross section.
\end{Ex}

\section{Lifting Commutation Relations}\label{Sec5}

This section contains the main results of this paper.

\paragraph{}
\begin{Def}
Let $A$ and $B$ be unital C*-algebras.  We say $A$ {\em absorbs $B$ smoothly} if there is an
isomorphism from $A$ to $A\otimes B$ which is homotopic to the embedding map 
$\bar\iota_A:a\mapsto a\otimes 1_B$ from $A$ to $A\otimes B$.
\end{Def}

If $B$ is smoothly self-absorbing, i.e.\ smoothly absorbs itself, then any $A$ that absorbs $B$
absorbs $B$ smoothly.  (It is an interesting and perhaps difficult question whether smoothly
self-absorbing is equivalent to strongly self-absorbing.)

Here are some examples of smooth absorption.  These are easy consequences of known (sometimes
deep) results, but have apparently not been explicitly written in the literature:
\begin{enumerate}
\item[(i)]  $O_2$ smoothly absorbs any separable simple unital nuclear C*-algebra since (a) it absorbs any
such algebra and (b) any unital endomorphism of $O_2$ is homotopic to the identity map by
\cite[2.1]{Cuntz} since the unitary group of $O_2$ is connected.
\item[(ii)]  $O_\infty$ is smoothly self-absorbing by the same argument (cf.\ \cite[8.2.3]{Rordam}); thus any $O_\infty$-absorbing
C*-algebra (e.g.\ any direct sum of Kirchberg algebras) absorbs $O_\infty$ smoothly.
\item[(iii)]  The Jiang-Su algebra $Z$ is smoothly self-absorbing by the results of \cite{DadarlatW}.
Thus any $Z$-stable C*-algebra absorbs $Z$ smoothly.
\end{enumerate}

\paragraph{}
\begin{Thm}\label{Thm40}
Let $A$ and $B$ be separable unital C*-algebras.  If $A$ is semiprojective and absorbs $B$
smoothly, then the quotient map from $A*B$ to $A\otimes B$ splits, and the quotient map from
$A*_{\C}B$ to $A\otimes B$ splits unitally.
\end{Thm}

\begin{proof}
Let $\phi$ be an isomorphism from $A$ to $A\otimes B$ which is homotopic to $\bar\iota_A$.  Since
$\bar\iota_A$ lifts to a homomorphism ($\iota_A$) from $A$ to $A*B$ and a unital homomorphism
from $A$ to $A*_{\C}B$, there is a lift $\psi$ of $\phi$ to $A*B$ and a unital lift $\tilde\psi$
of $\phi$ to $A*_{\C}B$ by the Homotopy Lifting Theorem \cite{BlackadarLifting}.  Then
$\psi\circ\phi^{-1}$ and $\tilde\psi\circ\phi^{-1}$ are the desired splittings.
\end{proof}

We give a lower-tech alternate argument which covers many of the same cases and also potentially
more situations.  In the following, let $A$ and $P$ be a separable unital C*-algebra with the following properties:
\begin{enumerate}
\item[(i)]    $P$ is isomorphic to $P^{\otimes\infty}=\displaystyle{\bigotimes_{k=1}^\infty P}$, and hence to 
$P^{\otimes n}=\displaystyle{\bigotimes_{k=1}^n P}$ for any $n$.
\item[(ii)]  $A$ is semiprojective and absorbs $P$, i.e.\ $A\otimes P\cong A$.
\end{enumerate}

Some examples of such $A$ and $P$ are:
\begin{enumerate}
\item[(a)]  $A=O_2$ or a finite direct sum of copies of $O_2$, $P=B^\infty$ for any separable simple
unital nuclear $B$, e.g.\ $P=O_2$, $O_\infty$, the Jiang-Su algebra $Z$, or the CAR algebra.
\item[(b)]  $A$ a finite direct sum of semiprojective Kirchberg algebras, e.g.\ $O_\infty$, $P=O_\infty$
or $P=Z$.
\end{enumerate}
It is likely there are some other possibilities, although the conditions on $A$ and $P$ are rather restrictive.

\begin{Thm}\label{Thm52}
Let $A$ and $P$ be as above.
Let $Q$ be the full free product of $A$ and a sequence of copies of $P$, i.e.\ $Q$ is the universal C*-algebra
generated by a copy of $A$ and a sequence of copies of $P$ with no relations.  The canonical quotient map 
$\pi:Q\to A\otimes P^{\otimes\infty}$ splits, i.e.\ there is a *-homomorphism $\sigma:A\otimes P^{\otimes\infty}\to Q$ with $\pi\circ\sigma=id$.
\end{Thm}

\begin{proof}
For each $n$, let
$$Q_n=(A\otimes P^{\otimes n})*P*P*\cdots$$
which is the universal C*-algebra generated 
by a copy of $A$ and a sequence of copies of $P$ such that the copy of $A$ and the first $n$ copies 
of $P$ commute and have a common unit.
There is an obvious canonical quotient map $\pi_n$ from $Q$ onto $Q_n$ for all $n$ and a
quotient map $\pi_{n,m}:Q_n\to Q_m$ for $n<m$ satisfying $\pi_{n,p}=\pi_{m,p}\circ\pi_{n,m}$
for $n<m<p$.  Thus we have an inductive system
$(Q_n,\pi_{n,m})$ with surjective connecting maps, and 
$$\limind(Q_n,\pi_{n,m})\cong A\otimes P^{\otimes\infty}$$
where the infinite tensor product is regarded as the universal C*-algebra generated by a copy
of $A$ and a sequence
of copies of $P$ which commute and have a common unit, and the isomorphism with the inductive
limit is the canonical one.

By assumption, $A\otimes P^{\otimes\infty}$ is isomorphic to $A$ and is thus semiprojective.
By definition of semiprojectivity, there is a lifting $\sigma:A\otimes P^{\otimes\infty}\to Q_n$ of the identity map on $A\otimes P^{\otimes\infty}$,
for some $n$.  Thus we have a cross section

\[%\UseTips
\xymatrix @=15mm {
{(A\otimes P^{\otimes n})*P*P*\cdots} \ar@<-1.pt>[r]_{\pi} & {(A\otimes P^{\otimes n})\otimes P\otimes P\otimes\cdots} 
\ar@<-3.pt>[l]_{\sigma} }\]
and the result follows by fixing an isomorphism of $A$ with $A\otimes P^{\otimes n}$.
\end{proof}

\begin{Cor}
Let $P$ and $A$ be as in the theorem.
For any $n\geq1$, the quotient map from $A*P*\cdots*P$ ($n$ copies of $P$) to $A\otimes P^{\otimes n}$ splits.
\end{Cor}

\begin{proof}
The quotient map $\pi$ from $Q$ to $A\otimes P^{\otimes\infty}$ factors through 
$$\rho:Q\to A*P*\cdots*P*P^{\otimes\infty}$$
($n-1$ copies of 
$P$), the universal C*-algebra
generated by a copy of $A$ and a sequence of copies of $P$ where the copies of $P$ for $k\geq n$ commute and have
a common unit.  If $\sigma:A\otimes P^{\otimes\infty}\cong A\otimes P^{\otimes(n-1)}\otimes P^{\otimes\infty}\to Q$ is the cross section from
the theorem, then $\rho\circ\sigma$ is the desired cross section once $P^{\otimes\infty}$ is identified with $P$.
\end{proof}

Since the quotient map from $Q$ to $A\otimes P^{\otimes\infty}$ factors through the infinite unital free product
of copies of $A$ and $P$, the theorem and corollary remain true if ``free product'' is replaced by
``unital free product.''  However, this does not directly give a unital splitting.  But we can simply repeat 
the argument replacing $Q$ by the unital infinite free product to obtain unital versions:

\begin{Thm}
Let $A$ and $P$ be as above.
Let $Q$ be the full unital free product of $A$ and a sequence of copies of $P$, i.e.\ $Q$ is the universal unital C*-algebra
generated by a copy of $A$ and a sequence of copies of $P$ all with a common unit but with no 
other relations.  The canonical quotient map 
$\pi:Q\to A\otimes P^{\otimes\infty}$ splits unitally, i.e.\ there is a unital *-homomorphism $\sigma:A\otimes P^{\otimes\infty}\to Q$ with $\pi\circ\sigma=id$.
\end{Thm}

\begin{Cor}
Let $P$ and $A$ be as in the theorem.
For any $n\geq1$, the quotient map from $A*_{\C}P*_{\C}\cdots*_{\C}P$ ($n$ copies of $P$) to $A\otimes P^{\otimes n}$ splits unitally.
\end{Cor}

\begin{Cor}
If $A$ and $P$ are as in the theorem, $B$ is a C*-algebra, $J$ a closed two-sided ideal of $B$ with
$B/J$ unital,
and $A_0$ and $P_0$ are commuting unital copies of $A$ and $P$ in $B/J$, and $A_0$ and $P_0$ lift
to copies of $A$ and $P$ in $B$, then $A_0$ and $P_0$ lift to commuting copies of $A$ 
and $P$ with a common unit in $B$.
If $B$ is unital,
and $A_0$ and $P_0$ lift
to unital copies of $A$ and $P$ in $B$, then $A_0$ and $P_0$ lift to commuting unital copies of $A$ 
and $P$ in $B$.
\end{Cor}

\begin{Cor}\label{KirchCor}
(i)  If $B$ is a separable simple unital nuclear C*-algebra, the canonical quotient map from
$O_2*B$ to $O_2\otimes B$ splits, and the quotient map from $O_2*_{\C}B$ to 
$O_2\otimes B$ splits unitally. 

\noindent
(ii) Let $A$ be a semiprojective Kirchberg algebra.  Then the canonical quotient map from
$A*O_\infty$ to $A\otimes O_\infty$ splits, and the quotient map from $A*_{\C}O_\infty$ to
$A\otimes O_\infty$ splits unitally. 

\noindent
(iii) Let $A$ be a $Z$-stable semiprojective C*-algebra.Then the canonical quotient map from
$A*Z$ to $A\otimes Z$ splits, and the quotient map from $A*_{\C}Z$ to
$A\otimes Z$ splits unitally.
\end{Cor}

Note that the only known unital $Z$-stable semiprojective C*-algebras are finite direct sums of
semiprojective Kirchberg algebras.

\section{How Explicit and General?}

The quotient maps from $O_2*O_n$ to $O_2\otimes O_n\cong O_2$ and from
 $O_n*O_\infty$
to $O_n\otimes O_\infty\cong O_n$ for any $n$, $2\leq n\leq\infty$, split.  It would be very useful to have
an explicit formula or description of a cross section in these cases (note that no uniqueness for the
splitting should be expected); but there may not be any such explicit
description, just as there is no explicit isomorphism between $O_2\otimes O_2$ and $O_2$.

There are still unresolved cases where both $A$ and $B$ are semiprojective Kirchberg algebras,
and some other interesting cases:
\begin{enumerate}
\item[(i)]  $O_m\otimes O_n$ where $m-1$ and $n-1$ are relatively prime and $>1$. 
Note that if $m-1$ and $n-1$ are not relatively prime, then there is no splitting for 
$O_m\otimes O_n$ (ruled out by Theorem \ref{Thm31}).
\item[(ii)]  $M_m(O_\infty)\otimes M_n(O_\infty)$, with $m,n>1$ and $m$ and $n$ relatively prime.
The case $m=1$ or $n=1$ is covered by the theorem, and the case $m$ and $n$ not relatively
prime is ruled out by Theorem \ref{Thm31}.
\item[(iii)]  Other cases potentially allowed by Theorem \ref{Thm31}(ii)--(iv).
\item[(iv)]  $O_2\otimes B$ with $B$ separable unital nuclear
but not simple, even the case $O_2\otimes\C^2$.  Note that the case
$O_2\otimes\M_2$ is covered by the theorem.
\item[(v)]  $O_\infty\otimes B$, where $B$ is separable unital nuclear but not purely infinite.
Even the cases $O_\infty\otimes\C^2$ and $O_\infty\otimes\M_2$ are open.
\end{enumerate}

The case where $A$ and $B$ are Kirchberg algebras which are not semiprojective is also open.
There are still severe $K$-theoretical obstructions to a splitting in this case, but there are some
instances not ruled out, e.g.\ if $A$ or $B$ is $O_\infty$, if $K_0(A)=K_0(B)=\Q$ or the dyadic rationals
($[1_A]$ or $[1_B]$ not 0) and $K_1(A)=K_1(B)=0$, or if the $K$-groups of $A$ and $B$ are in 
a case allowed in Theorem \ref{Thm31}(ii)--(iv) where the $G_i$ and $H_j$ are not required to
be finite, but just torsion with relatively prime exponents.
The question of existence of a splitting
would not seem to have any essential dependence on semiprojectivity (even though the present
proof techniques use semiprojectivity).  For example, in the UCT case one can write $A$ and $B$ as inductive
limits of semiprojective Kirchberg algebras $A_n$ and $B_n$ (cf.\ \cite{Rordam}, \cite{Enders}) and try to make splittings for
$A_n*B_n\to A_n\otimes B_n$ approximately compatible using approximate unitary equivalence
theorems for embeddings of Kirchberg algebras.

Cases such as $Z\otimes Z$, $Z\otimes\C^2$, $Z\otimes\M_2$, or $M_{2^\infty}\otimes M_{2^\infty}$, $M_{2^\infty}\otimes Z$, and $M_{2^\infty}\otimes\M_2$ where $M_{2^\infty}$ is the CAR algebra, are also open.
(A UHF algebra of infinite type is smoothly self-absorbing but not semiprojective, as is $Z$.)
Note that there is a $K$-theoretic obstruction to a splitting for $M_{2^\infty}\otimes\C^2$, but not the others.
\bibliography{o2splitref}

\begin{thebibliography}{EEL91}

\bibitem[AC13]{AraC}
Pere Ara and Guillermo Corti{\~n}as.
\newblock Tensor products of {L}eavitt path algebras.
\newblock {\em Proc. Amer. Math. Soc.}, 141(8):2629--2639, 2013.

\bibitem[Bla85]{BlackadarShape}
Bruce Blackadar.
\newblock Shape theory for {$C\sp \ast$}-algebras.
\newblock {\em Math. Scand.}, 56(2):249--275, 1985.

\bibitem[Bla98]{BlackadarKTheory}
Bruce Blackadar.
\newblock {\em {$K$}-theory for operator algebras}, volume~5 of {\em
  Mathematical Sciences Research Institute Publications}.
\newblock Cambridge University Press, Cambridge, second edition, 1998.

\bibitem[Bla04]{BlackadarSemiprojectivity}
Bruce Blackadar.
\newblock Semiprojectivity in simple {$C\sp *$}-algebras.
\newblock In {\em Operator algebras and applications}, volume~38 of {\em Adv.
  Stud. Pure Math.}, pages 1--17. Math. Soc. Japan, Tokyo, 2004.

\bibitem[Bla06]{BlackadarOperator}
B.~Blackadar.
\newblock {\em Operator algebras}, volume 122 of {\em Encyclopaedia of
  Mathematical Sciences}.
\newblock Springer-Verlag, Berlin, 2006.
\newblock Theory of $C{\sp{*}}$-algebras and von Neumann algebras, Operator
  Algebras and Non-commutative Geometry, III.

\bibitem[Bla15]{BlackadarLifting}
Bruce Blackadar.
\newblock The homotopy lifting theorem for semiprojective {$C\sp{\ast}
  $}-algebras.
\newblock {\em Math. Scand.}, to appear, 2015.

\bibitem[Cun81]{Cuntz}
Joachim Cuntz.
\newblock {$K$}-theory for certain {$C\sp{\ast} $}-algebras.
\newblock {\em Ann. of Math. (2)}, 113(1):181--197, 1981.

\bibitem[DW09]{DadarlatW}
Marius Dadarlat and Wilhelm Winter.
\newblock On the {$KK$}-theory of strongly self-absorbing {$C\sp *$}-algebras.
\newblock {\em Math. Scand.}, 104(1):95--107, 2009.

\bibitem[EE02]{EilersEFinite}
S{\o}ren Eilers and Ruy Exel.
\newblock Finite-dimensional representations of the soft torus.
\newblock {\em Proc. Amer. Math. Soc.}, 130(3):727--731 (electronic), 2002.

\bibitem[EEL91]{ElliottELSoft}
George~A. Elliott, Ruy Exel, and Terry~A. Loring.
\newblock The soft torus. {III}. {T}he flip.
\newblock {\em J. Operator Theory}, 26(2):333--344, 1991.

\bibitem[EL89]{ExelLAlmost}
Ruy Exel and Terry Loring.
\newblock Almost commuting unitary matrices.
\newblock {\em Proc. Amer. Math. Soc.}, 106(4):913--915, 1989.

\bibitem[End15]{Enders}
Dominic Enders.
\newblock Semiprojectivity for {K}irchberg algebras.
\newblock {\em preprint}, 2015.

\bibitem[Exe93]{ExelSoft}
Ruy Exel.
\newblock The soft torus and applications to almost commuting matrices.
\newblock {\em Pacific J. Math.}, 160(2):207--217, 1993.

\bibitem[Ger97]{Germain}
Emmanuel Germain.
\newblock {$KK$}-theory of the full free product of unital {$C\sp *$}-algebras.
\newblock {\em J. Reine Angew. Math.}, 485:1--10, 1997.

\bibitem[KP00]{KirchbergPEmbedding}
Eberhard Kirchberg and N.~Christopher Phillips.
\newblock Embedding of exact {$C\sp *$}-algebras in the {C}untz algebra
  {${\mathcal O}\sb 2$}.
\newblock {\em J. Reine Angew. Math.}, 525:17--53, 2000.

\bibitem[Lor97]{LoringLifting}
Terry~A. Loring.
\newblock {\em Lifting solutions to perturbing problems in {$C\sp
  *$}-algebras}, volume~8 of {\em Fields Institute Monographs}.
\newblock American Mathematical Society, Providence, RI, 1997.

\bibitem[R{\o}r94]{RordamShort}
Mikael R{\o}rdam.
\newblock A short proof of {E}lliott's theorem: {${\mathcal O}\sb
  2\otimes{\mathcal O}\sb 2\cong{\mathcal O}\sb 2$}.
\newblock {\em C. R. Math. Rep. Acad. Sci. Canada}, 16(1):31--36, 1994.

\bibitem[R{\o}r02]{Rordam}
M.~R{\o}rdam.
\newblock Classification of nuclear, simple {$C\sp *$}-algebras.
\newblock In {\em Classification of nuclear {$C\sp *$}-algebras. {E}ntropy in
  operator algebras}, volume 126 of {\em Encyclopaedia Math. Sci.}, pages
  1--145. Springer, Berlin, 2002.

\bibitem[RV98]{RordamVFree}
Mikael R{\o}rdam and Jesper Villadsen.
\newblock On the ordered {$K\sb 0$}-group of universal, free product {$C\sp
  *$}-algebras.
\newblock {\em $K$-Theory}, 15(4):307--322, 1998.

\bibitem[Spi09]{Spielberg}
Jack Spielberg.
\newblock Semiprojectivity for certain purely infinite {$C\sp *$}-algebras.
\newblock {\em Trans. Amer. Math. Soc.}, 361(6):2805--2830, 2009.

\bibitem[ST12]{SorensenTCharacterization}
Adam P.~W. S{\o}rensen and Hannes Thiel.
\newblock A characterization of semiprojectivity for commutative {$C\sp
  \ast$}-algebras.
\newblock {\em Proc. Lond. Math. Soc. (3)}, 105(5):1021--1046, 2012.

\bibitem[Szy02]{Szymanski}
Wojciech Szyma{\'n}ski.
\newblock On semiprojectivity of {$C\sp *$}-algebras of directed graphs.
\newblock {\em Proc. Amer. Math. Soc.}, 130(5):1391--1399 (electronic), 2002.

\bibitem[Voi83]{Voiculescu}
Dan Voiculescu.
\newblock Asymptotically commuting finite rank unitary operators without
  commuting approximants.
\newblock {\em Acta Sci. Math. (Szeged)}, 45(1-4):429--431, 1983.

\end{thebibliography}
\bibliographystyle{alpha}

\end{document}